\documentclass[12pt,reqno]{amsart}
\usepackage{fullpage}   % use as much space as possible
\usepackage{amssymb}    % AMS fonts
\usepackage{array}
\usepackage{enumitem}
\usepackage{url}
\usepackage{verbatim}

\theoremstyle{plain}
\newtheorem{theorem}{Theorem}%[section]
\newtheorem{lemma}[theorem]{Lemma}

\newtheorem{corollary}[theorem]{Corollary}

\theoremstyle{remark}
\newtheorem{example}[theorem]{Example}

\newtheorem*{acknowledgment}{Acknowledgment}

\newcommand{\thmlabel}[1]{\label{thm:#1}}   % theorem
\newcommand{\lemlabel}[1]{\label{lem:#1}}   % lemma
\newcommand{\corlabel}[1]{\label{cor:#1}}   % corollary
\newcommand{\eqnlabel}[1]{\label{eqn:#1}}   % equation

\newcommand{\thmref}[1]{\ref{thm:#1}}   % theorem
\newcommand{\lemref}[1]{\ref{lem:#1}}   % lemma
\newcommand{\eqnref}[1]{\eqref{eqn:#1}} % parenthesized eqn ref

\newcommand{\ldiv}{\backslash}
\newcommand{\rdiv}{/}
\newcommand{\inv}{^{-1}}
\newcommand{\AIP}{\textsc{AIP}}
\newcommand{\RIP}{\textsc{RIP}}
\newcommand{\WCIP}{\textsc{WCIP}}
\newcommand{\Sym}{\mathrm{Sym}}
\newcommand{\Atp}{\mathrm{Atp}}

\title{Pseudoautomorphisms of Bruck loops and their generalizations}

\author[M. Greer]{Mark Greer}
\email{\url{mark.greer@du.edu}}

\author[M. Kinyon]{Michael Kinyon}
\email{\url{mkinyon@du.edu}}

\address{Department of Mathematics\\
2360 S Gaylord St \\
University of Denver\\
Denver CO 80208 USA }


\begin{document}

\begin{abstract}
We show that in a weak commutative inverse property loop, such as a Bruck loop, if $\alpha$ is a right [left] pseudoautomorphism with companion $c$, then $c$ [$c^2$] must lie in the left nucleus. In particular, for any such loop with trivial left nucleus, every right pseudoautomorphism is an automorphism and if the squaring map is a permutation, then every left pseudoautomorphism is an automorphism as well. We also show that every pseudoautomorphism of a commutative inverse property loop is an automorphism, generalizing a well-known result of Bruck.
\end{abstract}

\maketitle

%\section{Introduction}
%\seclabel{intro}

A \emph{loop} $(Q,\cdot)$ consists of a set $Q$ with a binary operation $\cdot : Q\times Q\to Q$ such that (i) for all $a,b\in Q$, the equations $ax = b$ and $ya = b$ have unique solutions $x,y\in Q$, and (ii) there exists $1\in Q$ such that $1x = x1 = x$ for all $x\in Q$. We denote these unique solutions by $x = a\ldiv b$ and $y = b\rdiv a$, respectively.
For $x\in Q$, define the \emph{right} and \emph{left translations} by $x$ by, respectively, $y R_x = yx$ and $yL_x = xy$ for all $y\in Q$. That these mappings are permutations of $Q$ is essentially part of the definition of loop. Standard reference in loop theory are \cite{Br, Pf}.

A triple $(\alpha,\beta,\gamma)$ of permutations of a loop $Q$ is an \emph{autotopism} if for all $x,y\in Q$, $x\alpha\cdot y\beta = (xy)\gamma$. The set $\Atp(Q)$ of all autotopisms of $Q$ is a group under composition. Of particular interest here are the three subgroups
\begin{align*}
\Atp_{\lambda}(Q) &= \{(\alpha,\beta,\gamma)\in \Atp(Q)\mid 1\beta = 1\}\,, \\
\Atp_{\mu}(Q) &= \{(\alpha,\beta,\gamma)\in \Atp(Q)\mid 1\gamma = 1\}\,, \\
\Atp_{\rho}(Q) &= \{(\alpha,\beta,\gamma)\in \Atp(Q)\mid 1\alpha = 1\}\,.
\end{align*}
For instance, say, $(\alpha,\beta,\gamma)\in \Atp_{\lambda}(Q)$. For all $x\in Q$, $x\alpha = x\alpha\cdot 1 = x\alpha\cdot 1\beta = (x1)\gamma = x\gamma$. Thus $\alpha = \gamma$. Set $a = 1\alpha$. For all $x\in Q$, $x\alpha = (1x)\alpha = 1\alpha\cdot x\beta = a\cdot x\beta$ Thus $\alpha = \beta L_a$, and so every element of $\Atp_{\lambda}(Q)$ has the form $(\beta L_a,\beta,\beta L_a)$ for some $a\in Q$. Conversely, it is easy to see that if a triple of permutations of that form is an autotopism, then $1\beta = 1$.

By similar arguments for the other two cases, we have the following characterizations:
\begin{align*}
\Atp_{\lambda}(Q) &= \Atp(Q)\cap \{(\beta L_a,\beta,\beta L_a)\mid \beta\in \Sym(Q), a\in Q \}\,, \\
\Atp_{\mu}(Q) &= \Atp(Q)\cap \{(\gamma R_{c\ldiv 1}\inv ,\gamma L_c\inv,\gamma)\mid \gamma\in \Sym(Q), c\in Q\}\,, \\
\Atp_{\rho}(Q) &= \Atp(Q)\cap \{(\alpha,\alpha R_b,\alpha R_b)\mid \alpha\in \Sym(Q), b\in Q \}\,.
\end{align*}
Since these special types of autotopisms are entirely determined by a single permutation and an element of the loop, it is customary to focus on those instead of on the autotopisms themselves. This motivates the following definitions.

Let $Q$ be a loop. If $\beta\in \Sym(Q)$ and $a\in Q$ satisfy
\begin{equation}
\eqnlabel{leftpseudo}
a\cdot (xy)\beta = (a\cdot x\beta)(y\beta)
\end{equation}
for all $x,y\in Q$, then $\beta$ is called a \emph{left pseudoautomorphism} with \emph{companion} $a$. If $\gamma\in \Sym(Q)$ and $c\in Q$ satisfy
\begin{equation}
\eqnlabel{middlepseudo}
(xy)\gamma = [(x\gamma)\rdiv (c\ldiv 1)][c\ldiv (y\gamma)]
\end{equation}
for all $x,y\in Q$, then $\gamma$ is called a \emph{middle pseudoautomorphism} with \emph{companion} $c$. Finally, if $\alpha\in \Sym(Q)$ and $b\in Q$ satisfy
\begin{equation}
\eqnlabel{rightpseudo}
(xy)\alpha\cdot b = (x\alpha)(y\alpha\cdot b)
\end{equation}
for all $x,y\in Q$, then $\alpha$ is called a \emph{right pseudoautomorphism} with \emph{companion} $b$.

Pseudoautomorphisms can also be viewed as isomorphisms between loop isotopes where the isotopy is determined by the companion. Since this perspective will not play a role in what follows, we leave the details to the literature \cite{Br}.

There are some specializations of the notion of pseudoautomorphism worth mentioning explicitly. First, recall that the \emph{left}, \emph{middle} and \emph{right nucleus} of a loop $Q$ are the sets
\begin{align*}
N_{\lambda}(Q) &= \{ a\in Q\mid ax\cdot y = a\cdot xy,\ \forall x,y\in Q\}\,,\\
N_{\mu}(Q) &= \{ c\in Q\mid xc\cdot y = x\cdot cy,\ \forall x,y\in Q\}\,,\\
N_{\rho}(Q) &= \{ b\in Q\mid xy\cdot b = x\cdot yb,\ \forall x,y\in Q\}\,,
\end{align*}
respectively.

We denote the identity mapping on $Q$ by $\iota$.

\begin{lemma}
\lemlabel{nuclei}
Let $Q$ be a loop. The nuclei are characterized as follows:
\begin{align*}
N_{\lambda}(Q) &= \{ a\in Q\mid (\iota L_a, \iota, \iota L_a)\in \Atp(Q)\} \\
&= \{ a\in Q\mid \iota\text{ is a left pseudoautomorphism with companion } a\}\,,\\
N_{\mu}(Q) &= \{ c\in Q\mid (\iota R_c, \iota L_c\inv, \iota)\in \Atp(Q)\} \\
&= \{ c\in Q\mid \iota\text{ is a middle pseudoautomorphism with companion } c\}\,,\\
N_{\rho}(Q) &= \{ b\in Q\mid (\iota , \iota R_b, \iota R_b)\in \Atp(Q)\} \\
&= \{ b\in Q\mid \iota\text{ is a right pseudoautomorphism with companion } b\}\,.
\end{align*}
\end{lemma}

\begin{proof}
Perhaps the only claim which is not immediately obvious is the characterization of the middle nucleus. Suppose $\iota$ is a middle pseudoautomorphism with companion $c$. Then for all $x,y\in Q$, $xy = [x \rdiv (c\ldiv 1)][c\ldiv y]$. Replace $y$ with $cy$ to get $x\cdot cy = [x \rdiv (c\ldiv 1)] y$. Set $y = 1$ so that $xc = x \rdiv (c\ldiv 1)$. Thus $x\cdot cy = xc\cdot y$, that is, $c\in N_{\mu}(Q)$. The reverse inclusion is similarly straightforward.
\end{proof}

Note that all three of the nuclei are subloops. This can be proved directly from their definitions, but perhaps the easiest proof uses the autotopic characterization of Lemma \lemref{nuclei}.

A permutation $\sigma$ of a loop $Q$ is an \emph{automorphism} of $Q$ if $(xy)\sigma = (x\sigma )(y\sigma )$ for all $x,y\in Q$. Observe that a permutation $\sigma$ is an automorphism if and only if it is a pseudoautomorphism of any of the three types with companion $1$. The following is also clear from Lemma \lemref{nuclei}.

\begin{lemma}
\lemlabel{ps-aut}
Let $Q$ be a loop.
If $\sigma\in \Sym(Q)$ is a left [middle, right] pseudoautomorphism with companion $c\in Q$ then $\sigma$ is an automorphism if and only if $c\in N_{\lambda}(Q)$ [$N_{\mu}(Q)$, $N_{\rho}(Q)$].
\end{lemma}

\medskip

A loop $Q$ is said to be a (right) \emph{Bruck loop} if it satisfies the \emph{Bol identity} $[(xy)z]y = x[(yz)y]$ for all $x,y,z\in Q$ and the the \emph{automorphic inverse property} (AIP):
\[
(xy)\inv = x\inv y\inv  \tag{\AIP}
\]
for all $x,y\in Q$. (Bruck loops have also been called ``K-loops'' \cite{Kiechle} or
``gyrocommutative gyrogroups'' \cite{Ungar}. Note that much of the literature works with the dual notion of \emph{left} Bruck loop.) In a Bruck loop $Q$, inverses are two-sided, that is, $1\rdiv x = x\ldiv 1 = x\inv$, and the \emph{right inverse property} (RIP) holds:
\[
xy\cdot y\inv = x \qquad\text{or equivalently}\qquad R_y\inv = R_{y\inv}  \tag{\RIP}
\]
for all $x,y\in Q$. Bruck loops have been intensively studied in recent years \cite{Asch2,AKP,BS,BSS,BS2,Gl1,Kiechle,Nagy2}.

The interest in Bruck loops is partly because they are a naturally occurring class.
As an example, consider the set $S_n^+(\mathbb{R})$ of all $n\times n$ positive definite, symmetric matrices. By the polar decomposition, the product $AB$ of two such matrices decomposes uniquely as $AB = UP$ where $U$ is an orthogonal matrix and $P\in S_n^+(\mathbb{R})$. Define $A\odot B = P$. Then it is straightforward to show that $(S_n^+(\mathbb{R}),\odot)$ is a Bruck loop (see, \emph{e.g.}, \cite{Kiechle}).

Bruck loops are the motivation for our main result below, but we will state and prove it in much more generality (hence the generalizations mentioned in the title). The class of loops we will consider are those with two-sided inverses such that the following identity holds:
\[
(xy)\inv y = x\inv  \tag{\WCIP}
\]
for all $x,y$. These were introduced by Johnson and Sharma \cite{JS} who called them \emph{weak commutative inverse property loops} (WCIP loops). It is clear that any loop with the RIP and AIP satisfies WCIP. This applies in particular to Bruck loops or even to the more general class of Kikkawa loops \cite{Kiechle}. In fact, it is evident that any two of the properties RIP, AIP and WCIP imply the third.

\begin{lemma}
\lemlabel{wcip}
A loop $Q$ has the WCIP if and only if for all $x,y\in Q$,
\[
y\inv \ldiv x\inv = x\ldiv y\,.  \tag{{\WCIP}2}
\]
\end{lemma}

\begin{proof}
Replacing $y$ in (\WCIP) with $x\ldiv y$ and rearranging, we obtain ({\WCIP}2). Replacing $y$ in ({\WCIP}2) with $xy$ and rearranging, we obtain (\WCIP).
\end{proof}

In particular, Lemma \lemref{wcip} shows that a loop $Q$ has the WCIP if and only if the isotrophic loop \cite{Pf} $(Q,\circ)$ defined by $x\circ y = x\inv \ldiv y$ is commutative.

Before turning to our main result, we will show that in the present setting we can dispense with the notion of middle pseudoautomorphism. In a loop $Q$ with two-sided inverses, we will denote the inversion map by $J : Q\to Q; x\mapsto x\inv$.

\begin{lemma}
\lemlabel{atp-wcip}
Let $Q$ be loop with WCIP. If $(\alpha,\beta,\gamma)\in \Atp(Q)$, then $(J\gamma J, \beta, J\alpha J)\in \Atp(Q)$.
\end{lemma}

\begin{proof}
Since $(\alpha,\beta,\gamma)\in \Atp(Q)$, we have $x\alpha\cdot y\beta = (xy)\gamma$ for all $x,y\in Q$. Thus $(xy)\gamma J\cdot y\beta = (x\alpha\cdot y\beta)J\cdot y\beta = x\alpha J$ using the WCIP. Replace $x$ with $(xy)\inv$ and use the WCIP again to get $xJ\gamma J\cdot y\beta = (xy)J\alpha J$ for all $x,y\in Q$. Thus $(J\gamma J, \beta, J\alpha J)\in \Atp(Q)$.
\end{proof}

\begin{lemma}
\lemlabel{middle=right}
Let $Q$ be a loop with WCIP and let $\sigma\in \Sym(Q)$. Then $\sigma$ is a middle pseudoautomorphism with companion $c$ if and only if $J \sigma J$ is a right pseudoautomorphism with companion $c\inv$.
\end{lemma}

\begin{proof}
Suppose $\sigma$ is a middle pseudoautomorphism with companion $c$ so that
$(\sigma R_{c\inv}\inv ,\sigma L_c\inv,\sigma)$ is an autotopism. By Lemma \lemref{atp-wcip}, $(J\sigma J, \sigma L_c\inv ,J \sigma R_{c\inv}\inv J)\in \Atp(Q)$. Since the first component fixes $1$, this autotopism lies in $\Atp_{\rho}(Q)$, and so the second and third components coincide and have the form $J\sigma J R_d$ for some $d$. To determine $d$, we compute $d = 1 J\sigma R_{c\inv}\inv J = c\inv$. Thus $(J \sigma J,J\sigma J R_c,J\sigma J R_c )\in \Atp_{\rho}(Q)$, that is, $\sigma$ is a right pseudoautomorphism with companion $c\inv$. The converse is similar.
\end{proof}

As an aside, we mention that a similar result holds for loops with the right inverse property: $\sigma$ is a middle pseudoautomorphism with companion $c$ if and only if $\sigma$ is a right pseudoautomorphism with companion $c$. In place of Lemma \lemref{atp-wcip}, the argument uses the fact that in RIP loops, $(\alpha,\beta,\gamma)\in \Atp(Q)$ implies $(\gamma,J\beta J,\alpha)\in \Atp(Q)$ \cite{Kiechle}.

As a corollary of Lemmas \lemref{nuclei} and \lemref{middle=right}, we re-obtain a fact from \cite{JS}.

\begin{corollary}
\corlabel{Nm=Nr}
In a loop $Q$ with WCIP, $N_{\mu}(Q) = N_{\rho}(Q)$.
\end{corollary}

Our main result is the following.

\begin{theorem}
\thmlabel{main}
Let $Q$ be a WCIP loop, let $\sigma$ be a permutation of $Q$ and let $c\in Q$.
\begin{enumerate}
\item If $\sigma$ is a right pseudoautomorphism of $Q$ with companion $c$, then
$c\in N_{\lambda}(Q)$.
\item If $\sigma$ is a left pseudoautomorphism of $Q$ with companion $c$, then $c\inv$ is also a companion of $\sigma$ and $c^2 \in N_{\lambda}(Q)$.
\end{enumerate}
\end{theorem}

\begin{proof}
(1) Since $1 = yy\inv = y\cdot x(x\ldiv y\inv)$, we have
\[
c = 1\sigma\cdot c = y\sigma \cdot ((x(x\ldiv y\inv))\sigma \cdot c) =
y\sigma \cdot [x\sigma\cdot ((x\ldiv y\inv)\sigma \cdot c)]\,.
\]
Thus
\begin{equation}
\eqnlabel{wcip-tmp1}
x\sigma\ldiv (y\sigma \ldiv c) = (x\ldiv y\inv)\sigma \cdot c\,.
\end{equation}
Exchanging the roles of $x$ and $y$, we also have
\begin{equation}
\eqnlabel{wcip-tmp2}
y\sigma\ldiv (x\sigma \ldiv c) = (y\ldiv x\inv)\sigma \cdot c\,.
\end{equation}
By ({\WCIP}2), the right sides of \eqnref{wcip-tmp1} and \eqnref{wcip-tmp2} are equal, and so
\begin{equation}
\eqnlabel{wcip-tmp3}
x\sigma\ldiv (y\sigma \ldiv c) = y\sigma\ldiv (x\sigma \ldiv c)\,.
\end{equation}
Replacing $x$ with $x\sigma\inv$ and $y$ with $y\sigma\inv$ in \eqnref{wcip-tmp6}, we have $x\ldiv (y \ldiv c) = y\ldiv (x \ldiv c)$, and so
\begin{equation}
\eqnlabel{wcip-tmp4}
x (y\ldiv (x \ldiv c)) = y \ldiv c\,.
\end{equation}
Setting $x = c$ in \eqnref{wcip-tmp4}, we obtain
\begin{equation}
\eqnlabel{wcip-tmp5}
y\ldiv c = cy\inv\,.
\end{equation}
Using \eqnref{wcip-tmp5} in \eqnref{wcip-tmp4}, we have
\begin{equation}
\eqnlabel{wcip-tmp6}
x(y\ldiv (cx\inv)) = cy\inv\,.
\end{equation}
Taking $y = cx\inv$ in \eqnref{wcip-tmp6}, we get
\begin{equation}
\eqnlabel{wcip-tmp7}
c(cx\inv)\inv = x\,.
\end{equation}
Now in \eqnref{wcip-tmp6}, replace $x$ with $cx\inv$ and use \eqnref{wcip-tmp7} and ({\WCIP}2) to obtain
\begin{equation}
\eqnlabel{wcip-tmp8}
cx\inv\cdot (x\inv\ldiv y\inv) = cy\inv\,.
\end{equation}
Finally, in \eqnref{wcip-tmp8}, replace $x$ with $x\inv$ and $y$ with $y\inv$, and then replace $y$ with $xy$ to get
\[
cx\cdot y = c\cdot xy\,,
\]
which shows $c\in N_{\lambda}(Q)$, as claimed.

(2) Since $(\sigma L_c, \sigma,\sigma L_c)\in \Atp(Q)$, we have $(J\sigma L_c J,\sigma,J\sigma L_c J)\in \Atp(Q)$ by Lemma \lemref{wcip}. Since $1\sigma = 1$, this autotopism lies in $\Atp_{\lambda}(Q)$. Thus $J\sigma L_c J = \sigma L_d$ where $d = 1J\sigma L_c J = c\inv$. Hence $(\sigma L_{c\inv},\sigma,\sigma L_{c\inv})\in \Atp(Q)$, which shows that $\sigma$ has $c\inv$ as a companion. We have
\[
(L_{c\inv}\inv \sigma\inv,\sigma\inv,L_{c\inv}\inv \sigma\inv)(\sigma L_c,\sigma,\sigma L_c) = (L_{c\inv}\inv L_c, \iota, L_{c\inv}\inv L_c)\in \Atp(Q)\,.
\]
Therefore $L_{c\inv}\inv L_c = L_e$ where $e = 1L_{c\inv}\inv L_c = c^2$. Thus $(L_{c^2},\iota,L_{c^2})\in \Atp(Q)$, that is, $c^2\in N_{\lambda}(Q)$.
\end{proof}

\begin{corollary}
Let $Q$ be a WCIP loop with trivial left nucleus. Then every right pseudoautomorphism
is an automorphism. If, in addition, every element of $Q$ has a unique square root, then
every left pseudoautomorphism is an automorphism.
\end{corollary}

\begin{example}
The \emph{relativistic Bruck loop} (or relativistic gyrocommutative gyrogroup) is the set of relativistic velocity vectors with Einstein's velocity addition as the operation \cite{Ungar}. This is isomorphic to the natural Bruck loop structure on the set of positive definite symmetric Lorentz transformations \cite[Ch. 10]{Kiechle}. The left nucleus is trivial, because it is precisely the set of fixed points of the action of the special orthogonal group. In addition, every element of the loop has a unique square root. Thus we obtain: \emph{In the relativistic Bruck loop, every pseudoautomorphism is an automorphism}.
\end{example}

Finally, we generalize a well-known result of Bruck \cite{BrPs}, who proved the following for commutative Moufang loops.

\begin{corollary}
Every pseudoautomorphism of a commutative, inverse property loop is an automorphism.
\end{corollary}

\begin{proof}
In an inverse property loop, all nuclei coincide, so by Theorem \thmref{main} and its left/right dual, the companion of any pseudoautomorphism lies in the nucleus of $Q$. By Lemma \lemref{ps-aut}, we have the desired result.
\end{proof}

\begin{acknowledgment}
Our investigations were aided by the automated deduction tool \textsc{Prover9} developed by McCune \cite{McCune}. The problem of the existence of pseudoautomorphisms of the relativistic Bruck loop which are not automorphisms was suggested to the second author several years ago by Anton Greil.
\end{acknowledgment}

\end{document}